\documentclass[11pt,reqno]{amsart}
\usepackage[a4paper,margin=27mm]{geometry}
\usepackage{amsmath,amssymb,amsthm,mathtools}
\numberwithin{equation}{section}
\usepackage[hidelinks]{hyperref}
\newcommand{\R}{\mathbb R}
\newcommand{\Sph}{\mathbb S}
\newcommand{\cH}{\mathcal H}
\newcommand{\cB}{\mathcal B}
\newcommand{\cC}{\mathcal C}
\newcommand{\eps}{\varepsilon}
\newcommand{\supp}{\operatorname{supp}}
\newcommand{\Lip}{\operatorname{Lip}}

\theoremstyle{plain}
\newtheorem{theorem}{Theorem}[section]
\newtheorem{lemma}[theorem]{Lemma}
\newtheorem{corollary}[theorem]{Corollary}
\theoremstyle{remark}

\title[Homogeneous polynomial approximation]{Homogeneous polynomial approximation on symmetric convex bodies: A proof of Kro\'o's conjecture}
\author{Matthew J. Colbrook}
\address{Department of Applied Mathematics and Theoretical Physics,
University of Cambridge, Centre for Mathematical Sciences,
Wilberforce Road, Cambridge CB3 0WA, United Kingdom}
\email{m.colbrook@damtp.cam.ac.uk}
\date{October 24, 2026}

\begin{document}
\begin{abstract}
We resolve Kroó's 2004 conjecture on homogeneous polynomial approximation in every dimension, without smoothness or curvature assumptions. For every origin-symmetric convex body \(K\subset\R^d\) with nonempty interior, every continuous real function on \(\partial K\) is uniformly approximated by sums of homogeneous polynomials of degrees \(n\) and \(n+1\) as \(n\to\infty\). Moreover, approximation of the constant function by homogeneous polynomials of degree \(2n\) is uniform over all such bodies in each fixed dimension. Via homogenization, we also resolve Kro\'o's 2019 weighted approximation conjecture for every convex admissible weight on Euclidean space. The key construction produces homogeneous polynomials with a common Gaussian boundary profile and localization errors controlled by arbitrary polynomial weights. Its analytic core is a weighted Gaussian approximation theorem: for \(d\ge2\), sufficiently broad Gaussians on \(\mathbb R^{d-1}\) can be approximated by \(e^{-f}F\), where \(f\) is nonnegative, convex and Lipschitz with \(f(0)=0\), and \(F\) is a Schwartz function with Fourier support in an arbitrarily small ball about the origin. Small Fourier support allows Taylor truncation to a lower degree, while the remainder enforces decay through homogenization. We then use positive spherical quadrature and degree raising to yield the full approximation theorem.
\end{abstract}

\maketitle

\section{Introduction}

The classical Weierstrass theorem permits uniform approximation by polynomials whose terms may have many different degrees. We study the corresponding problem when each approximant is restricted to two consecutive homogeneous degrees. Homogeneity imposes a parity restriction: polynomials of even degree are even, and those of odd degree are odd. Kro\'o's conjecture asserts that two consecutive degrees suffice on the boundary of every origin-symmetric convex body.

The Euclidean norm and inner product are denoted by \(|\cdot|\) and
\(\cdot\), respectively, and \(\Sph^{d-1}=\{x\in\R^d:|x|=1\}\).
For an integer \(N\geq0\), let \(\cH_N^d\) be the real vector space
spanned by monomials of total degree \(N\) on \(\R^d\) (\(\cH_0^d\) is the space of constants). For a norm \(\nu\), define
\[
K_\nu=\{x\in\R^d:\nu(x)\leq1\},
\qquad
x_\nu(u)=\frac{u}{\nu(u)}\quad(u\in\Sph^{d-1}).
\]
Every origin-symmetric convex body with nonempty interior is the unit ball of
its Minkowski norm
\[
\nu_K(x)=\inf\{t>0:x\in tK\}.
\]
The notation \(C(E)\) denotes the real continuous
functions on a compact set \(E\), with the supremum norm. In the local
construction, \(s=d-1\geq1\).
For \(u,v\in\Sph^s\), define
\[
\rho(u,v)=\min\{|u-v|,|u+v|\}.
\]
The following theorem proves the conjecture and a uniform version for
bounded norm classes.

\begin{theorem}[Homogeneous polynomial approximation]\label{thm:main}
For every origin-symmetric convex body \(K\subset\R^d\) with nonempty
interior, there are \(h_{2n}\in\cH_{2n}^d\) such that
\[
\|h_{2n}-1\|_{C(\partial K)}\rightarrow0.
\]
Moreover, for fixed \(d\) and \(0<\alpha\leq\beta<\infty\),
\[
\lim_{n\to\infty}
\sup_{\substack{\nu\ {\rm a\ norm}\\
\alpha|x|\leq\nu(x)\leq\beta|x|\ \forall x}}
\inf_{h\in\cH_{2n}^d}
\|h-1\|_{C(\partial K_\nu)}=0.
\]
Consequently, for every \(f\in C(\partial K)\),
\[
\lim_{n\to\infty}\inf_{p\in\cH_n^d+\cH_{n+1}^d}
\|f-p\|_{C(\partial K)}=0.
\]
If \(f\) is even(odd), the approximants may be chosen in \(\cH_{2n}^d\) (\(\cH_{2n+1}^d\)).
\end{theorem}

The following corollary follows from the uniform assertion by a linear change of
coordinates. It shows that the degree needed to approximate the constant function
to a fixed accuracy depends only on the dimension.

\begin{corollary}[Uniformity over convex bodies]\label{cor:allbodies}
Let \(\mathcal K_d\) be the class of origin-symmetric convex bodies in
\(\R^d\) with nonempty interior. Then
\[
\lim_{n\to\infty}\sup_{K\in\mathcal K_d}
\inf_{h\in\cH_{2n}^d}\|h-1\|_{C(\partial K)}=0.
\]
\end{corollary}

Theorem~\ref{thm:main} also resolves the weighted approximation conjecture
in \cite[Conjecture~2]{Kroo2019Weighted}.\footnote{In one variable, Totik and Varj\'u~\cite{TotikVarju2007} relate
approximation with varying weights to the local regularity of the
associated equilibrium measure.} For \(m\geq1\), a positive
continuous even weight \(w\) on \(\R^m\) is admissible if
\(t\,w(x/t)\) is nondecreasing for \(t>0\), for every fixed \(x\), and
converges as \(t\downarrow0\), uniformly for \(|x|=1\), to a positive
continuous function. For every convex admissible weight and every
continuous \(g:\R^m\to\R\) such that
\[
g(ru)\rightarrow g_\infty(u)\qquad(r\to\infty)
\]
uniformly for \(u\in\Sph^{m-1}\), with \(g_\infty\) continuous and even,
there are real polynomials \(p_{2n}\) of degree at most \(2n\) satisfying
\[
\sup_{x\in\R^m}
\left|w(x)^{-2n}p_{2n}(x)-g(x)\right|\rightarrow0.
\]
Indeed, the continuous extension of
\(w^*(x,t)=|t|w(x/t)\), \(t\ne0\), is a norm, and
\(K_w=\{w^*\leq1\}\) is an origin-symmetric convex body.
The function \(G(x,t)=g(x/t)\) extends continuously and evenly to
\(\partial K_w\), with \(G(x,0)=g_\infty(x/|x|)\).
Applying Theorem~\ref{thm:main} to \(G\) and taking
\(p_{2n}(x)=h_{2n}(x,1)\) gives the assertion by homogeneity.
This is the correspondence established in
\cite[Lemmas~1--2]{Kroo2019Weighted}; it removes the regularity assumption
in that paper's Theorem~1.
The results of this paper have also been formally verified in Lean \cite{palomar-2026-10-27-000004-v1}.

\subsection{Previous results}

Kro\'o formulated this conjecture in 2004; see
\cite{BenkoKroo2009} and \cite[Section~3 and Problem~1]{Kroo2023}.
Early results of Kro\'o and Szabados~\cite{KrooSzabados2005} included
the crosspolytope case. The planar case was established independently
by Benko and Kro\'o~\cite{BenkoKroo2009} and Varj\'u~\cite{Varju2007}.
Benko and Kro\'o also treated convex bodies with \(C^{1+\gamma}\)
boundary, \(\gamma>0\), while Varj\'u proved the result for polytopes
and for \(C^2\) boundaries with positive curvature. Kro\'o and
Szabados~\cite{KrooSzabados2008} obtained Jackson-type estimates under a
Dini--Lipschitz condition on the outer normal. They subsequently extended
the result to regular convex bodies, namely those with a unique supporting
hyperplane at every boundary point~\cite{KrooSzabados2009}. They also
proved the corresponding \(L^p\) approximation theorem on every symmetric
convex body, for \(1\leq p<\infty\). Totik~\cite{Totik2013} gave an
elementary proof of the planar and regular cases. Further cases,
including all origin-symmetric convex bodies of revolution in \(\R^3\),
were obtained by Kro\'o~\cite[Corollary~9]{Kroo2019Star}.
Kro\'o's survey~\cite{Kroo2023} describes these developments and the
general open problem.\footnote{A proof of the full conjecture was claimed
by Gonz\'alez Merino and Villa~\cite{GonzalezMerinoVilla2020}. The preprint
was withdrawn in July 2021; the withdrawal notice identifies its
Theorem~1.3 and Corollary~1.4 as incorrect.}
Theorem~\ref{thm:main} resolves this problem in every dimension.

Although smooth
bodies or polytopes can approximate every symmetric convex body, this does not by itself transfer homogeneous
approximation to the limit. A relative radial change of size \(\delta\)
multiplies the value of a degree-\(N\) homogeneous polynomial by
\((1+\delta)^N\). Thus the degree needed for approximation interacts
with the geometric error, and the known results for individual
approximating bodies do not supply the required uniform control.
Our argument treats the nonsmooth local profiles directly and gives
estimates uniform over normalized norm classes.

Barvinok~\cite{Barvinok2003} approximates arbitrary norms by roots of
nonnegative homogeneous polynomials. Ahmadi, de Klerk and
Hall~\cite{AhmadiDeKlerkHall2019} obtain approximation by roots of convex
positive definite forms and study optimisation over these polynomial
norms by semidefinite programming. Such root approximation does not by
itself control the polynomials on the unit boundary. By homogeneity, the approximants in Theorem~\ref{thm:main} give bounded
sublevel sets \(\{x:h_{2n}(x)\leq1\}\) for sufficiently large \(n\),
converging radially to \(K\). Lasserre~\cite{Lasserre2015} studies
minimum-volume containment in homogeneous polynomial sublevel sets,
extending the L\"owner--John ellipsoid problem.

Totik's construction~\cite[Section~2]{Totik2013} approximates the constant
function by summing localized homogeneous polynomials. Localization also
enters extremal problems. Kro\'o~\cite{Kroo2015} studied ordinary and
homogeneous needle polynomials, relating their decay to the boundary
geometry and applying them to norming sets and cubature formulas.
His later work~\cite{Kroo2026} constructs homogeneous needle polynomials
with small integral norms and applies them to Christoffel functions,
Markov inequalities, and discretization.

Localized polynomial kernels and positive quadrature also underlie the
construction of tight frames on spheres by Narcowich, Petrushev and
Ward~\cite{NarcowichPetrushevWard2006}.

Theorem~\ref{thm:kernel} supplies a common Gaussian profile on the
boundary, with a discrepancy controlled by arbitrary polynomial
weights at the localization scale. The estimates hold uniformly over
all centers and all convex bodies in a normalized norm class.
The integral of this prescribed profile against spherical measure
in the directional variable is independent of the center and the body.
This permits a common normalization when the local polynomials are
summed. The construction handles arbitrary nonsmooth convex profiles
through the weighted approximation and finite-family arguments of
Lemmas~\ref{lem:weighted} and~\ref{lem:finite}.

\subsection{The mechanism}

The proof of Theorem~\ref{thm:main} constructs localized polynomials of a prescribed degree.
Their restrictions to the boundary approximate Gaussians centered at an
antipodal pair. Polynomial decay of the error allows us to sum these
approximations without losing control as the degree increases.

\begin{theorem}[Uniform Gaussian approximation]\label{thm:kernel}
Fix \(s\geq1\), \(0<\alpha\leq\beta<\infty\), \(L\geq0\), and
\(\eps>0\). There are \(A\geq1\) and an integer \(N_0\geq2\), depending only on these
parameters, such that for every norm \(\nu\) satisfying
\[
\alpha|x|\leq\nu(x)\leq\beta|x|
\qquad(x\in\R^{s+1}),
\]
every \(v\in\Sph^s\), and every even \(N\geq N_0\), there is a real
\(R_{N,v}\in\cH_N^{s+1}\) for which
\[
\sup_{u\in\Sph^s}
\left(1+\frac{N}{A}\rho(u,v)\right)^L
\left|
R_{N,v}(x_\nu(u))
-\exp\left(-\frac{N^2}{A^2}\rho(u,v)^2\right)
\right|<\eps.
\]
\end{theorem}

The analytic mechanism is independent of a convex body's geometry.
Lemma~\ref{lem:weighted} approximates broad Gaussians after multiplication
by \(e^{-f}\), for every nonnegative Lipschitz convex function \(f\)
vanishing at zero. Supporting inequalities provide one-sided exponential
decay in the directions spanned by gradients of \(f\). The Gaussian
Fourier cutoff handles directions where \(f\) is constant.
The error bound is independent of these directions and of the supporting
slopes. We then use compactness to obtain a finite family of approximants for
each uniformly Lipschitz class.

For a convex body, coordinates in a supporting hyperplane express the
radial factor through a function \(f_N\) that we approximate in this class. Small Fourier support permits
Taylor approximation using only half of the available degree and the
remaining degree gives decay through a power of the supporting
functional. The prescribed
Gaussian profile then allows the same spherical normalization at every
center. Positive quadrature combines the local approximations into
a polynomial approximating the constant function.

\subsection{Outline}

Section~\ref{sec:weights} proves the one-dimensional weighted density
lemma and its tensor-product form. Section~\ref{sec:convex} establishes
Gaussian approximation with convex weights and the selection of a finite
family. Section~\ref{sec:kernels} proves Theorem~\ref{thm:kernel}.
Section~\ref{sec:completion} completes the approximation theorem and the
uniform corollary.

\section{One-sided exponential weights and small bandwidth}\label{sec:weights}

We first isolate the weighted approximation that compensates for the
radial factor introduced by homogeneity. Exponential decay of the weight
on one half-line permits approximation with arbitrarily small Fourier
support. This observation will apply to supporting slopes of a convex
function.

All Fourier and weighted function spaces below have complex scalars. Let
\(C_0(\R^r)\) denote the continuous complex-valued functions that vanish
at infinity, and \(B(0,\sigma)\) is the open Euclidean ball.
Our Fourier convention is
\[
\widehat h(\xi)=\int_{\R^r}h(x)e^{-ix\cdot\xi}\,dx,
\qquad
\mathcal F^{-1}\phi(x)=(2\pi)^{-r}
\int_{\R^r}\phi(\xi)e^{ix\cdot\xi}\,d\xi.
\]
For \(\sigma>0\), let
\[
\cB_\sigma(\R^r)
=\{\mathcal F^{-1}\phi:
\phi\in C_c^\infty(B_{\R^r}(0,\sigma))\}.
\]
Every member is a Schwartz function. Taking real parts preserves this
space and does not increase the error. In dimension
zero, \(\cB_\sigma(\R^0)\) denotes the constants on a point.

\begin{lemma}[Approximation with a one-sided exponential weight]\label{lem:onesided}
Let \(c>0\), \(M\geq0\), and \(\delta>0\), and define
\[
W(t)=(1+|t|)^M e^{-ct_+},
\qquad t_+=\max(t,0).
\]
Then \(\cB_\delta(\R)\) is dense in
\[
X_W=\{h\in C(\R):Wh\in C_0(\R)\},
\qquad
\|h\|_{X_W}=\|Wh\|_\infty.
\]
\end{lemma}

\begin{proof}
Multiplication by \(W\) is an isometric isomorphism from \(X_W\) onto
\(C_0(\R)\). If density fails, the Hahn--Banach and Riesz representation
theorems give a nonzero finite complex
measure \(\lambda\) such that
\[
\int W(t)F(t)\,d\lambda(t)=0
\qquad(F\in\cB_\delta(\R)).
\]
Choose a Schwartz function \(q\) with \(q(0)=1\) and
\(\supp\widehat q\subset(-1,1)\), and define \(q_\eta(t)=q(\eta t)\).
For \(0<\eta<\delta\), define the finite complex measure
\[
d\mu_\eta(t)=q_\eta(t)W(t)\,d\lambda(t).
\]
Its finiteness follows because \(q_\eta\) is Schwartz and \(W\) grows
at most polynomially. If \(|\xi|<\delta-\eta\), the modulated function
\(e^{-i\xi t}q_\eta(t)\) belongs to \(\cB_\delta(\R)\). The annihilation
hypothesis therefore gives
\[
\widehat{\mu_\eta}=0
\quad\hbox{on }(-\delta+\eta,\delta-\eta).
\]
Furthermore,
\[
\int e^{ct}\,d|\mu_\eta|(t)
\leq
\sup_t\{(1+|t|)^M|q_\eta(t)|\}\,\|\lambda\|<\infty,
\]
since \(ct-ct_+\leq0\).

Thus
\[
H_\eta(z)=\int e^{-izt}\,d\mu_\eta(t)
\]
is holomorphic on \(0<\operatorname{Im}z<c\) and continuous on the
real boundary. The exponential moments at heights \(0\) and \(c\)
give the required domination. On a compact substrip, the positive
distance from its two boundary lines also controls every derivative
under the integral.
The boundary values vanish on a nonempty real interval. On a smaller open
interval, the reflected function \(\overline{H_\eta(\overline z)}\)
agrees with the zero boundary values. The Schwarz reflection principle
therefore gives a holomorphic extension across that interval. The
identity theorem implies that \(H_\eta\) vanishes throughout the strip
and, by continuity, on the real axis. Fourier uniqueness for finite complex
measures gives \(\mu_\eta=0\).

On every fixed compact interval, \(q_\eta\) has no zeros for all sufficiently
small \(\eta\), since \(q(0)=1\). Because \(W\) is positive and continuous,
the identity \(q_\eta W\lambda=0\) implies that \(\lambda\) vanishes on
that compact interval. Hence \(\lambda=0\), a contradiction.
\end{proof}

\begin{lemma}[Tensor-product density]\label{lem:tensor}
For \(1\leq i\leq r\), let \(W_i>0\) be continuous on \(\R\), and let
\(E_i\) be dense in \(X_{W_i}\). Finite sums of products
\[
\prod_{i=1}^r F_i(t_i),\qquad F_i\in E_i,
\]
are dense in the space of continuous \(H\) for which
\[
\left(\prod_iW_i(t_i)\right)H(t)\in C_0(\R^r),
\]
with the corresponding weighted supremum norm.
\end{lemma}

\begin{proof}
Multiplication by the product weight identifies the target space with
\(C_0(\R^r)\). Finite sums of products of one-dimensional \(C_0\) functions
are dense in this space. To see this, first truncate a target to a compact box, then
approximate it there using a finite rectangular partition of unity.
Each one-dimensional factor can in turn be approximated by \(W_iF_i\),
and a telescoping product estimate controls the error.
\end{proof}

\section{Gaussian approximation with convex weights}\label{sec:convex}

We next apply the one-dimensional density result in the directions
spanned by gradients of a convex function. The function is constant in
the remaining directions, where a Fourier cutoff provides the Gaussian
factor. We begin with the cutoff estimate.

For \(A\geq1\), define
\[
w_A(x)=(1+|x|/A)^L,
\qquad G_A(x)=e^{-|x|^2/A^2},
\qquad
C_L=\sup_{r\geq0}(1+r)^Le^{-r^2}<\infty.
\]

\begin{lemma}[Fourier cutoff of a Gaussian]\label{lem:cutoff}
Fix \(s\), \(L\geq0\), and \(\sigma>0\). For each \(0\leq k\leq s\)
there are real \(g_{A,k}\in\cB_{\sigma/4}(\R^k)\) such that
\[
e_k(A):=
\sup_{z\in\R^k}(1+|z|/A)^L
|g_{A,k}(z)-e^{-|z|^2/A^2}|\rightarrow0
\]
as \(A\to\infty\). In dimension zero set \(g_{A,0}=1\) and
\(e_0(A)=0\). The convergence is uniform over these finitely many
dimensions.
\end{lemma}

\begin{proof}
For \(k\geq1\), choose a real radial \(\chi_k\in C_c^\infty(B(0,1))\)
equal to one on \(B(0,1/2)\), and set
\[
g_{A,k}
=\mathcal F^{-1}\left[
\chi_k(4\xi/\sigma)\,
A^k\pi^{k/2}e^{-A^2|\xi|^2/4}\right].
\]
After the change of variables \(z=Ay\), the function
\(y\mapsto e^{-|y|^2}-g_{A,k}(Ay)\) is the inverse Fourier transform of
\[
\left[1-\chi_k\left(\frac{4\zeta}{\sigma A}\right)\right]
\pi^{k/2}e^{-|\zeta|^2/4}.
\]
For an integer \(m>L\), the weighted supremum of this inverse transform
is bounded by a constant times the sum of the \(L^1\) norms of its Fourier
derivatives of orders at most \(m\). This follows by integration by parts.
Every such derivative tends to zero in \(L^1\): the uncut Gaussian term
is restricted to \(|\zeta|\geq\sigma A/8\), and derivatives of the cutoff
have additional nonpositive powers of \(A\) and the same Gaussian-tail
support. Hence \(e_k(A)\to0\). Taking the maximum over the finitely many
\(k\) proves the uniform assertion.
\end{proof}

\begin{lemma}[Gaussian approximation with a convex weight]\label{lem:weighted}
Fix \(s\geq1\), \(L\geq0\), and \(\sigma>0\). There is a function
\(\Delta(A)\to0\) as \(A\to\infty\), depending only on these parameters,
such that every globally Lipschitz convex \(f:\R^s\to[0,\infty)\) with
\(f(0)=0\) satisfies
\[
\inf_{F\in\cB_\sigma(\R^s)}
\sup_x w_A(x)|e^{-f(x)}F(x)-G_A(x)|\leq\Delta(A).
\]
The approximants may be chosen real.
\end{lemma}

\begin{proof}
Let \(S\) be the linear span of the gradients of \(f\) at its
differentiability points, and let \(r=\dim S\).
The function \(f\) is constant on translates of \(S^\perp\).
Indeed, the directional distributional derivatives in \(S^\perp\) vanish,
because the almost-everywhere gradient belongs to \(S\). Convolving with
a smooth compactly supported approximate identity gives smooth functions
with the same zero derivatives, hence constant on those fibers. Their
locally uniform limit has the same property.

Suppose first that \(r>0\). Choose linearly independent gradients
\(a_1,\ldots,a_r\) spanning \(S\), attained at differentiability points
\(y_1,\ldots,y_r\). Define
\[
C_i=a_i\cdot y_i-f(y_i).
\]
The supporting inequalities and \(f(0)=0\) give \(C_i\geq0\) and
\[
f(x)\geq a_i\cdot x-C_i.
\]
On \(S\), use the invertible linear coordinates
\[
t_i=a_i\cdot y,\qquad 1\leq i\leq r,
\qquad t=Ty.
\]
Since \(f\geq0\),
\[
(t_i)_+\leq f(y)+C_i.
\]
Consequently
\[
f(y)\geq\frac1r\sum_{i=1}^r(t_i)_+-C,
\qquad C=\frac1r\sum_{i=1}^r C_i.
\]
The inverse coordinate map and
\(1+\sum b_i\leq\prod(1+b_i)\), for \(b_i\geq0\), give a finite constant \(C_f\) such
that
\begin{equation}\label{eq:productdomination}
(1+|y|/A)^Le^{-f(y)}
\leq C_f\prod_{i=1}^r
\left[(1+|t_i|)^Le^{-(t_i)_+/r}\right].
\end{equation}
The constant \(C_f\) may depend on \(f\), but is independent of \(A\geq1\).

The function to be approximated in these coordinates is
\[
H(t)=e^{f(T^{-1}t)}e^{-|T^{-1}t|^2/A^2}.
\]
It belongs to the product weighted \(C_0\) space in
\eqref{eq:productdomination}. The exponent \(f(T^{-1}t)\) grows at most
linearly, whereas the Gaussian exponent is a negative definite quadratic.

By Lemmas~\ref{lem:onesided} and \ref{lem:tensor}, \(H\) can be approximated
arbitrarily well in that product weighted norm by finite sums of products
of one-dimensional bandlimited Schwartz functions. Their one-dimensional
bandwidths can be arbitrarily small. Choose a common bandwidth
\(\delta>0\) with
\[
\delta\sum_{i=1}^r|a_i|<\sigma/4.
\]
After returning to the \(y\) coordinates, each such finite sum is a
Schwartz function \(F_S\) on \(S\), with smooth Fourier transform supported
in
\[
\left\{\sum_i\omega_i a_i:|\omega_i|<\delta\right\}
\subset B_S(0,\sigma/4).
\]
By \eqref{eq:productdomination}, for every \(\eta>0\) we may arrange
\begin{equation}\label{eq:activeerror}
\sup_{y\in S}(1+|y|/A)^L
|e^{-f(y)}F_S(y)-e^{-|y|^2/A^2}|<\eta.
\end{equation}

Decompose \(x=y+z\), where \(y\in S\) and \(z\in S^\perp\), and let
\(k=s-r\). Let \(g_{A,k}\) be supplied by Lemma~\ref{lem:cutoff}, using
an orthonormal identification of \(S^\perp\) with \(\R^k\). Then
\[
F(y+z)=F_S(y)g_{A,k}(z)
\]
belongs to \(\cB_\sigma(\R^s)\): the two Fourier supports lie in orthogonal
balls of radius \(\sigma/4\), and their product lies strictly inside
\(B(0,\sigma)\).

Using
\[
w_A(y+z)\leq(1+|y|/A)^L(1+|z|/A)^L
\]
and splitting the error between the two factors gives
\[
\|w_A(e^{-f}F-G_A)\|_\infty
\leq
\eta\,\|(1+|z|/A)^Lg_{A,k}(z)\|_\infty+C_Le_k(A).
\]
The first coefficient is finite. Letting \(\eta\downarrow0\) proves the
desired infimum bound with
\[
\Delta(A)=C_L\max_{0\leq k\leq s}e_k(A).
\]

If \(r=0\), constancy on translates of \(S^\perp\) implies \(f=0\).
Take \(F=g_{A,s}\). If \(r=s\), the complementary factor is the
dimension-zero constant, and the infimum is zero by
\eqref{eq:activeerror}. In both cases the stated bound follows.
Finally, taking real parts does not increase error against the real target.
\end{proof}

\subsection{A finite family of approximants}

The preceding lemma controls the approximation error independently of the
convex function. To obtain uniform polynomial estimates later, we also
need to control the approximants themselves. Compactness gives a finite
family from which every approximant can be selected.

For \(C>0\), let
\[
\cC_C=
\{f:\R^s\to[0,\infty):
f\text{ convex},\ f(0)=0,\ \Lip(f)\leq C\}.
\]

\begin{lemma}[A finite family for a Lipschitz class]\label{lem:finite}
Fix \(s\geq1\), \(L\geq0\), \(\sigma>0\), \(C>0\), and \(\tau>0\). There are \(A\geq1\) and finitely
many real functions
\[
F_1,\ldots,F_J\in\cB_\sigma(\R^s)
\]
such that every \(f\in\cC_C\) has an index \(j\) satisfying
\[
\sup_x w_A(x)|e^{-f(x)}F_j(x)-G_A(x)|<\tau.
\]
\end{lemma}

\begin{proof}
The class \(\cC_C\) is compact for locally uniform convergence. To see this, note that
on every bounded ball its members are uniformly bounded and equicontinuous,
and convexity, nonnegativity, the value at zero, and the Lipschitz bound
persist under a locally uniform limit. Arzelà--Ascoli and a diagonal
argument give sequential compactness. The locally uniform topology is
metrizable, for example by
\[
d_{\mathrm{loc}}(f,g)=\sum_{m=1}^{\infty}2^{-m}
\min\{1,\|f-g\|_{C(\overline B(0,m))}\},
\]
so the class is compact.

Choose \(A\geq1\) so that \(\Delta(A)<\tau/4\) in
Lemma~\ref{lem:weighted}. Each \(f\in\cC_C\) then has a real
\(F\in\cB_\sigma\) with error below \(\tau/2\).

For a fixed \(F\), the error functional
\[
f\longmapsto\|w_A(e^{-f}F-G_A)\|_\infty
\]
is continuous for locally uniform convergence on \(\cC_C\).
Indeed, the difference of two such errors is at most
\[
\|w_AF(e^{-f}-e^{-g})\|_\infty.
\]
On a fixed ball, \(|e^{-f}-e^{-g}|\leq|f-g|\), since \(f,g\geq0\).
Outside a sufficiently large ball, \(w_AF\) is uniformly small, while
\(|e^{-f}-e^{-g}|\leq2\).

The open neighbourhoods on which the chosen \(F\) has error below \(\tau\)
cover \(\cC_C\). A finite subcover gives the required finite family.
\end{proof}

\section{Homogeneous Gaussian approximation on the boundary}\label{sec:kernels}

We now transfer the convex-weighted approximation to the boundary of the
unit ball. Coordinates in a supporting hyperplane give a nonnegative
convex function with a uniform Lipschitz bound. The finite family from
Lemma~\ref{lem:finite} provides a function \(F_j\) for which
\(e^{-f_N}F_j\) approximates the Gaussian. Small Fourier support permits
Taylor truncation at half the prescribed degree.

\begin{proof}[Proof of Theorem~\ref{thm:kernel}]
Define
\[
D=\beta/\alpha,\qquad
C=D(1+D),\qquad
\sigma_0=\frac1{64eD},\qquad
B_L=\max(1,\sqrt2D)^L.
\]
Choose
\[
0<\tau<\min\{1/4,\eps/(8B_L)\}.
\]
Apply Lemma~\ref{lem:finite} with Lipschitz bound \(C\) and bandwidth
\(\sigma_0\). This determines \(A\geq1\) and a finite family of real functions
\(F_1,\ldots,F_J\), before any norm, center, or degree is chosen.

\subsection{Coordinates from a supporting hyperplane}

Fix an admissible norm \(\nu\) and \(v\in\Sph^s\), and let
\(a_0=\nu(v)\), so \(\alpha\leq a_0\leq\beta\). Choose a supporting
linear functional \(\ell\) with
\[
\ell(v)=a_0,\qquad |\ell(x)|\leq\nu(x).
\]
In particular \(\|\ell\|\leq\beta\).

Let \(U:\R^s\to v^\perp\) be an isometry, and let \(U^*:\R^{s+1}\to\R^s\)
denote its adjoint. Define
\[
Bt=Ut-\frac{\ell(Ut)}{a_0}v.
\]
Then
\[
\ell(Bt)=0,\qquad U^*Bt=t,\qquad \|B\|\leq1+D.
\]
For each even \(N\geq2\), set
\begin{equation}\label{eq:fn}
f_N(x)=N\left(\frac{\nu(v+Bx/N)}{a_0}-1\right).
\end{equation}
The supporting inequality gives \(f_N\geq0\), and \(f_N(0)=0\).
It is convex and has Lipschitz constant at most \(D(1+D)=C\).
Therefore \(f_N\in\cC_C\). Select an index \(j\) for which
\begin{equation}\label{eq:selected}
\|w_A(e^{-f_N}F_j-G_A)\|_\infty<\tau.
\end{equation}
The selected function may depend on \(\nu,v,N\), but it always belongs to the
same fixed finite family.

For a point \(z=x_\nu(u)\in\partial K_\nu\) with \(a=\ell(z)>0\), define
\[
t=\frac{a_0U^*z}{a}.
\]
These coordinates satisfy
\begin{equation}\label{eq:chartidentities}
z=\frac{a}{a_0}(v+Bt),\qquad
a=\frac{a_0}{\nu(v+Bt)}
=\left(1+\frac{f_N(Nt)}N\right)^{-1}.
\end{equation}

On the chart \(a\geq1/2\), the direction \(u\) satisfies
\(\ell(u)\geq\nu(u)/2\geq\alpha/2\). Hence
\[
t=\frac{a_0U^*u}{\ell(u)},\qquad |t|\leq2D.
\]
The elementary chord--sine relations
\[
|U^*u|\leq\rho(u,v)\leq\sqrt2\,|U^*u|
\]
give
\begin{equation}\label{eq:anglecomparison}
\frac{|t|}{2D}\leq\rho(u,v)\leq\sqrt2D\,|t|.
\end{equation}
Thus the angular weight on this chart is at most \(B_Lw_A(Nt)\).

For bounded \(x=Nt\), the angular map has the uniform expansion
\[
\frac{v+Bt}{|v+Bt|}=v+Ut+O_D(|t|^2).
\]
Consequently, for each fixed \(M\),
\begin{equation}\label{eq:localangle}
\sup_{|x|\leq M}
\left|N\rho\left(\frac{v+Bx/N}{|v+Bx/N|},v\right)-|x|\right|
 =O_D(M^2/N),
\end{equation}
uniformly over the geometric choices.

\subsection{Taylor approximation and homogenization}

Absorbing the Fourier normalization into \(\phi_j\) gives
\[
F_j(x)=\int_{|\xi|<\sigma_0}\phi_j(\xi)e^{i\xi\cdot x}\,d\xi,
\qquad \phi_j\in C_c^\infty(B(0,\sigma_0)).
\]
Let
\[
C_\phi=\max_j\|\phi_j\|_1<\infty.
\]
For even \(N\geq2\), let \(m=N/2\), and define the polynomial of total degree
at most \(m\)
\[
P_{N,j}(t)=
\int\phi_j(\xi)\sum_{k=0}^{m}\frac{(iN\xi\cdot t)^k}{k!}\,d\xi.
\]
Its coefficients are real: they are the Taylor coefficients of the real
function \(t\mapsto F_j(Nt)\) at zero.

On \(|t|\leq T=2D\), the exponential Taylor remainder gives
\[
|P_{N,j}(t)-F_j(Nt)|
\leq C_\phi e^{\sigma_0NT}
\frac{(\sigma_0NT)^{m+1}}{(m+1)!}.
\]
Here \(\sigma_0T=1/(32e)\). Using \(n!\geq(n/e)^n\) and
\(m+1>N/2\), we obtain, for every even \(N\geq2\),
\begin{equation}\label{eq:taylor}
\sup_{|t|\leq2D}|P_{N,j}(t)-F_j(Nt)|
\leq C_\phi e^{-N}.
\end{equation}
Indeed \(2e\sigma_0T=1/16\), and
\(\log4-1/(32e)>1\).

Define
\begin{equation}\label{eq:homogenisation}
R_{N,j}(z)=\ell(z)^N
P_{N,j}\left(\frac{a_0U^*z}{\ell(z)}\right).
\end{equation}
Because \(\deg P_{N,j}\leq N/2\), this rational expression extends
across \(\ell(z)=0\) as a real homogeneous polynomial of degree \(N\).
Since \(N\) is even, \(R_{N,j}\) is even. At \(z=v/a_0\), its
value is \(F_j(0)\), which by \eqref{eq:selected} is within \(\tau<1/4\)
of one. Thus \(R_{N,j}\) is nonzero and has degree exactly \(N\).
For the index selected in \eqref{eq:selected}, define
\(R_{N,v}=R_{N,j}\).

\subsection{The estimate near the supporting hyperplane}

For \(z\in\partial K_\nu\), set \(b=a_0U^*z\). Then \(|b|\leq D\).
If \(0\leq a=\ell(z)\leq1/2\), expansion of
\eqref{eq:homogenisation} gives
\[
|R_{N,j}(z)|
\leq C_\phi\,2^{-(N-m)}
\sum_{k=0}^{m}\frac{(N\sigma_0D)^k}{k!}
\leq C_\phi\,2^{-N/2}e^{N\sigma_0D}
\leq C_\phi e^{-N/4}.
\]
By evenness, precisely the same estimate holds when
\(-1/2\leq a\leq0\). Thus
\begin{equation}\label{eq:outside}
|R_{N,j}(z)|\leq C_\phi e^{-N/4}
\qquad\text{when }|\ell(z)|\leq1/2.
\end{equation}
Evenness reduces the region \(\ell(z)\leq-1/2\) to the positive chart.

The function
\[
h(u)=\ell(u)/\nu(u),\qquad u\in\Sph^s,
\]
is \(2D\)-Lipschitz. To see this, split the difference of two quotients
and use \(|\ell(w)|\leq\nu(w)\), \(\Lip(\nu)\leq\beta\), and
\(\nu\geq\alpha\) on the sphere. Since \(h(v)=1\) and \(h(-v)=-1\),
the condition \(|\ell(x_\nu(u))|\leq1/2\) implies
\begin{equation}\label{eq:separation}
\rho(u,v)\geq\frac1{4D}.
\end{equation}
Therefore, the Gaussian target is exponentially small in \(N^2\) on the
same region. The angular error weight is at most
\((1+\sqrt2N/A)^L\); hence \eqref{eq:outside} and
\eqref{eq:separation} make its entire weighted error tend uniformly to zero.

\subsection{The estimate near the center}

Assume now \(a\geq1/2\) and let \(x=Nt\). For \(|x|\leq M\), let
\(q=f_N(x)\), so \(0\leq q\leq CM\). From
\(\log(1+t)\geq t-t^2/2\) and \(\log(1+t)\leq t\), \(t\geq0\), we get
\[
1\leq
\left(1+\frac qN\right)^{-N}e^q
\leq \exp\left(\frac{q^2}{2N}\right).
\]
Let this ratio be \(r_N(x)\). By \eqref{eq:selected},
\[
w_A(x)|e^{-f_N(x)}F_j(x)|\leq C_L+\tau.
\]
Thus
\begin{align*}
w_A(x)|a^NF_j(x)-G_A(x)|
&\leq \tau+(r_N(x)-1)(C_L+\tau)\\
&\leq
\tau+\left(e^{C^2M^2/(2N)}-1\right)(C_L+\tau).
\end{align*}
This bound is uniform in \(\nu,v,N,j\) once \(M\) is fixed.

The Taylor error \eqref{eq:taylor}, even after multiplication by the
angular weight, tends uniformly to zero. Also
\eqref{eq:localangle} makes the angular Gaussian converge uniformly to
\(G_A(x)\) on \(|x|\leq M\). For example, the derivative of
\(r\mapsto e^{-r^2/A^2}\) has magnitude at most \(\sqrt{2/e}/A\).
Using \eqref{eq:anglecomparison}, the weighted error on this region is
therefore bounded by
\begin{equation}\label{eq:centralerror}
B_L\tau+o(1),
\end{equation}
where the \(o(1)\) is uniform over all norms, centers, and selected functions.

\subsection{The estimate away from the center}

If \(a\geq1/2\) and \(|x|=|Nt|>M\), then \(a^N\leq1\), so
\eqref{eq:taylor} gives
\[
|R_{N,j}(z)|\leq |F_j(x)|+C_\phi e^{-N}.
\]
By \eqref{eq:anglecomparison}, the Gaussian target is at most
\[
\exp\left(-\frac{|x|^2}{4D^2A^2}\right),
\]
and the angular weight is at most \(B_Lw_A(x)\).
The finite family of Schwartz functions therefore gives
\[
\lim_{M\to\infty}
\max_j\sup_{|x|>M}
w_A(x)\left[
|F_j(x)|+
\exp\left(-\frac{|x|^2}{4D^2A^2}\right)\right]=0.
\]
The Taylor remainder, multiplied by the maximal angular weight
\((1+\sqrt2N/A)^L\), also tends uniformly to zero.

Choose \(M\geq1\), after the fixed finite family, so that the last supremum
multiplied by \(B_L\) is less than \(\eps/4\). Then choose \(N_0\)
so large that the weighted error on \(|\ell(z)|\leq1/2\) is less than
\(\eps/2\), the remainder in \eqref{eq:centralerror} is less than
\(\eps/4\), and the weighted Taylor remainder on the tail is less than
\(\eps/4\). The central error is then less than \(3\eps/8\), and
the tail error is less than \(\eps/2\). All three regions therefore
satisfy the required strict error bound \(<\eps\).
Points with \(\ell(z)\leq-1/2\) are reduced to the positive
chart by evenness. All choices are uniform in the norm and center,
so Theorem~\ref{thm:kernel} follows.
\end{proof}

\section{Completion of the approximation theorem}\label{sec:completion}

The Gaussian integral on the sphere is independent of its center. We
approximate this integral by a positive finite sum and replace each
Gaussian by the corresponding homogeneous polynomial. The weighted
error estimate in Theorem~\ref{thm:kernel} controls the resulting sum.

\subsection{Spherical quadrature}

For \(\lambda>0\), define
\[
G_\lambda(u,v)=e^{-\lambda^2\rho(u,v)^2},
\qquad
T_{\lambda,L}(u,v)=(1+\lambda\rho(u,v))^{-L}.
\]
Let \(d\omega\) be surface measure on \(\Sph^s\).

\begin{lemma}[Positive spherical quadrature]\label{lem:quadrature}
If \(L>s\) and \(\lambda\geq1\), the rotation-invariant quantities
\[
I_\lambda=\int_{\Sph^s}G_\lambda(u,v)\,d\omega(v),
\qquad
J_{\lambda,L}=\int_{\Sph^s}T_{\lambda,L}(u,v)\,d\omega(v)
\]
are independent of \(u\), and satisfy
\[
0<c_s\lambda^{-s}\leq I_\lambda\leq C_s\lambda^{-s},
\qquad
J_{\lambda,L}\leq C_{s,L}\lambda^{-s},
\]
where the positive constants depend only on their subscripts.
For every fixed \(\lambda\) and \(\zeta>0\), there are finitely many
centres \(v_i\) and positive weights \(w_i\) such that, uniformly in \(u\),
\[
\left|\sum_iw_iG_\lambda(u,v_i)-I_\lambda\right|<\zeta I_\lambda,
\qquad
\sum_iw_iT_{\lambda,L}(u,v_i)\leq2J_{\lambda,L}.
\]
\end{lemma}

\begin{proof}
Rotation invariance gives independence of \(u\). For a nonnegative
continuous function \(k\), spherical coordinates and the substitution
\(r=2\sin(\theta/2)\) give
\[
\int_{\Sph^s}k(\rho(u,v))\,d\omega(v)
=2|\Sph^{s-1}|\int_0^{\sqrt2}
k(r)r^{s-1}(1-r^2/4)^{(s-2)/2}\,dr.
\]
The final factor is bounded above and below by positive constants
depending only on \(s\). Substitution of \(t=\lambda r\) proves both
upper bounds, since
\(\int_0^\infty t^{s-1}(1+t)^{-L}\,dt<\infty\) when \(L>s\).
Integration over \(0\leq r\leq1/\lambda\) gives the lower bound.

For the quadrature, both kernels are jointly continuous on the compact
product of spheres. Partition the sphere into finitely many sets of
sufficiently small diameter, choose a center in each positive-area set,
and use its surface measure as the weight. Uniform continuity makes the
two Riemann-sum errors uniformly as small as required in \(u\).
Choose their absolute tolerances below \(\zeta I_\lambda\) and
\(J_{\lambda,L}\), respectively.
\end{proof}

\subsection{Approximation of the constant function}

\begin{proof}[Proof of Theorem~\ref{thm:main}: the constant function]
For \(d=1\), every body is an interval \([-a,a]\), and
\(h_{2n}(x)=(x/a)^{2n}\) equals one on its boundary.
The uniform assertion is therefore exact in dimension one.
Assume henceforth that \(d\geq2\), so \(s=d-1\geq1\).
Fix \(L>s\). Lemma~\ref{lem:quadrature} gives a constant \(B\) with
\[
\frac{2J_{\lambda,L}}{I_\lambda}\leq B
\qquad(\lambda\geq1).
\]
Given a desired error \(\eta>0\), apply Theorem~\ref{thm:kernel} with
kernel error \(\eps<\eta/(2B)\). It gives a fixed \(A\) and a uniform
degree threshold. For every sufficiently large even \(N\), define
\(\lambda=N/A\geq1\), and choose the quadrature with \(\zeta<\eta/2\).
At its finitely many centers, choose degree-\(N\) kernels \(R_{N,v_i}\).
The polynomial
\[
H_N(x)=I_\lambda^{-1}\sum_iw_iR_{N,v_i}(x)
\]
is homogeneous of degree \(N\), and for \(x=x_\nu(u)\in\partial K_\nu\),
\[
|H_N(x)-1|
\leq\zeta+
\eps\,\frac{2J_{\lambda,L}}{I_\lambda}<\eta.
\]
The constants and the threshold are uniform over the prescribed norm class,
which proves the uniform assertion in Theorem~\ref{thm:main}.

Every individual convex body belongs to such a norm-equivalence class.
For that body, choose approximants with errors tending to zero along all
even degrees, using increasing thresholds for errors \(1/j\).
At each degree, the construction requires a choice of polynomial at
only finitely many centers.
\end{proof}

\subsection{Uniformity over convex bodies}

\begin{proof}[Proof of Corollary~\ref{cor:allbodies}]
Choose \(e_1,\ldots,e_d\in K\) maximizing
\(|\det(e_1,\ldots,e_d)|\), and let \(E\) be the matrix with these
columns. Compactness gives a maximizer, and the maximum is positive
because \(K\) has nonempty interior. Replacing column \(i\) by
\(x\in K\) shows that \(|(E^{-1}x)_i|\leq1\). Symmetry and convexity
therefore imply
\[
\{y:\textstyle\sum_{i=1}^d|y_i|\leq1\}
\subset E^{-1}K\subset[-1,1]^d.
\]
The norm \(\widetilde\nu(y)=\nu_K(Ey)\) consequently satisfies
\[
d^{-1/2}|y|\leq\widetilde\nu(y)\leq d^{1/2}|y|.
\]
Composition with \(E\) is a bijection on \(\cH_{2n}^d\) and preserves
the corresponding boundary approximation errors. The uniform assertion
of Theorem~\ref{thm:main}, with \(\alpha=d^{-1/2}\) and
\(\beta=d^{1/2}\), proves the result.
\end{proof}

\subsection{Approximation of continuous functions}

Approximation of the constant function allows lower-degree homogeneous
terms to be raised to a common degree. Polynomial approximation and
symmetry then give all even functions. Multiplication by coordinate
functions gives the odd case.

\begin{proof}[Completion of Theorem~\ref{thm:main}]
First suppose \(f\in C(\partial K)\) is even.
Ordinary polynomial approximation on the compact set \(\partial K\),
followed by symmetrization, gives an even polynomial
\[
p=\sum_{j=0}^m p_{2j},
\qquad p_{2j}\in\cH_{2j}^d,
\]
arbitrarily close to \(f\).
Let \(h_{2n}\to1\) be the sequence just constructed.
For \(n>m\), the polynomial
\[
P_{2n}^{(p)}=\sum_{j=0}^m p_{2j}h_{2(n-j)}
\]
is homogeneous of degree \(2n\) and converges uniformly to \(p\) as
\(n\to\infty\). Choosing increasingly accurate \(p\)'s and then increasing
thresholds gives approximants to \(f\) along every sufficiently large
even degree.

If \(f\) is odd, define
\[
g_j(x)=\frac{x_jf(x)}{|x|^2},\qquad 1\leq j\leq d,
\quad x\in\partial K.
\]
These are continuous and even, because \(0\notin\partial K\), and
\[
f(x)=\sum_{j=1}^d x_jg_j(x).
\]
Approximate each \(g_j\) by homogeneous polynomials of degree \(2n\) and
multiply by \(x_j\). Their sum is homogeneous of degree \(2n+1\) and
converges uniformly to \(f\).

Finally, decompose an arbitrary \(f\) into its even and odd parts.
For \(n=2m\), use the even approximation in degree \(2m\) and the odd
approximation in degree \(2m+1\). For \(n=2m+1\), use the odd approximation
in degree \(2m+1\) and the even approximation in degree \(2m+2\).
Both sums converge uniformly to \(f\).
\end{proof}




\noindent\textbf{AI declaration.} The author formulated the research questions and mathematical ideas and directed the investigation. During the development of the proofs, the author interacted with and guided ChatGPT 6, which suggested proof strategies and intermediate arguments. The author critically evaluated these suggestions, developed and verified the proofs, and wrote the manuscript. ChatGPT 6, accessed through Codex, also assisted with the formalisation in Lean. The author takes full responsibility for the mathematical results, proofs, and final manuscript.

\end{document}